\documentclass[12pt]{amsart}
\usepackage{amsmath}
\usepackage{amsfonts}
\usepackage{amssymb}
\usepackage{mathtools}
\usepackage{amsthm}
\usepackage{eucal}
\usepackage{url}
\numberwithin{equation}{section}

\newtheorem{Theorem}{Theorem}[section]

\newtheorem{Lemma}[Theorem]{Lemma}

\newtheorem{Proposition}[Theorem]{Proposition}
\newtheorem{definition}[Theorem]{Definition}
\theoremstyle{remark}

\theoremstyle{remark}
\newtheorem*{remark}{Remark}
\newcommand{\codim}{\mathrm{codim}}
\theoremstyle{definition}
\newtheorem{Example}[Theorem]{Example}
\usepackage{mathrsfs}
\usepackage{eucal}
\usepackage{graphicx}
\renewcommand{\tilde}{\widetilde}

\newcommand{\cc}{\mathbb{C}}
\newcommand{\pp}{\mathbb{P}}
\renewcommand{\gg}{\mathbb{G}}

\newcommand{\E}{\mathcal{E}}

\newcommand{\G}{\mathscr{G}}
\newcommand{\cF}{\mathcal{F}}

\renewcommand{\L}{\mathcal{L}}

\newcommand{\N}{\mathcal{N}}

\renewcommand{\O}{\mathcal{O}}
\newcommand{\op}{\O_{\pp^1}}

\DeclareMathOperator{\spec}{Spec}
\usepackage{multicol,lipsum}
\newcommand\undermat[2]{%
  \makebox[0pt][l]{$\smash{\underbrace{\phantom{%
    \begin{matrix}#2\end{matrix}}}_{\text{$#1$}}}$}#2}
\usepackage[utf8]{inputenc}
\usepackage{tikz}
\usepackage{tikz-cd}
\usetikzlibrary{matrix,arrows,decorations.pathmorphing}
\usepackage[margin=1.4in]{geometry}
\title{Normal bundles of lines on hypersurfaces}
\author{Hannah K. Larson}
\address{Department of Mathematics, Harvard University, One Oxford Street, Cambridge MA 02138}
\email{hannahlarson@college.harvard.edu}

\begin{document}

\maketitle

\begin{abstract}
Let $X \subset \pp^n$ be a smooth hypersurface. Given a sequence of integers $\vec{a} = (a_1, \ldots, a_{n-2})$ with $a_1 \leq \cdots \leq a_{n-2}$, let $F_{\vec{a}}(X)$ be the parameter space of lines $L$ on $X$ such that $N_{L/X} \cong \O(a_1) \oplus \cdots \oplus \O(a_{n-2})$. The loci $F_{\vec{a}}(X)$ form a stratification of the Fano scheme of lines on $X$. We show that for general hypersurfaces, the $F_{\vec{a}}(X)$ have the expected dimension and, in this case, compute the class of $\overline{F_{\vec{a}}(X)}$ in the Chow ring of the Grassmannian of lines in $\pp^n$. For certain splitting types $\vec{a}$, we also provide non-trivial upper bounds on the dimension of $F_{\vec{a}}(X)$ that hold for all smooth $X$.
\end{abstract}

\section{Introduction}
Let $X \subset \pp^n$ be a smooth hypersurface of degree $d$ over the complex numbers. This paper is concerned with the geometry of the \textit{Fano scheme} parameterizing lines on $X$, defined by
\[F(X) := \{L \in \gg(1, n): L \subset X\},\]
where $\gg(1, n)$ is the Grassmannian parameterizing lines in $\pp^n$. Let $N = {n+d \choose d} - 1$ and let $U \subset \pp^N$ be the open subset parameterizing smooth hypersurfaces of degree $d$ in $\pp^n$. We define the \textit{universal Fano scheme} to be
\[\Sigma := \{(L, X) \in \gg(1, n) \times U : L \subset X\}.\]
The condition that a hypersurface contain a fixed line is $d+1$ linear conditions on the parameter space $\pp^N$, so looking at the projection $\Sigma \rightarrow \gg(1, n)$, one readily sees that $\Sigma$ is smooth and irreducible of dimension
\[\dim \Sigma = N-d - 1 + \dim \gg(1, n) = N + 2n - d - 3.\]

This dimension count gives rise to an expected dimension of $2n - d - 3$ for the Fano scheme of lines on hypersurface of degree $d$ in $\pp^n$. It is a well-known result that this expected dimension is indeed achieved for general hypersurfaces. A prominent conjecture of Debarre-de Jong states that the expected dimension should be achieved for all smooth hypersurfaces of degree $d \leq n$. The Debarre-de Jong Conjecture has been proved for $d \leq 8$ by Beheshti \cite{B} and for $d \ll n$ by Harris et al. \cite{H}.

Here, we study the normal bundles $N_{L/X}$ of lines $L$ in $X$, which govern the local geometry of $F(X)$ at $L$. For each $L \subset X \subset \pp^n$, there is a short exact sequence of normal bundles
\[0 \rightarrow N_{L/X} \rightarrow N_{L/\pp^n} \rightarrow N_{X/\pp^n}|_L \rightarrow 0,\]
in which the middle term is $\O(1)^{n-1}$ and the rightmost term is $\O(d)$. Since every vector bundle on $\pp^1$ splits as a direct sum of line bundles, it follows that 
\[N_{L/X} \cong \O(a_1) \oplus \cdots \oplus \O(a_{n-2})\]
for integers $a_1 \leq  \cdots \leq a_{n-2} \leq 1$ with $a_1 + \ldots + a_{n-2} = n - d - 1$.
To study the behavior of the normal bundle, for each sequence of integers $\vec{a} =  (a_1, \ldots, a_{n-2})$ satisfying the above conditions, we define
\[F_{\vec{a}}(X) := \{L \in F(X): N_{L/X} \cong \O(a_1) \oplus \cdots \oplus \O(a_{n-2})\}\]
and its universal counterpart
\[\Sigma_{\vec{a}} := \{(L, X) \in \Sigma : N_{L/X} \cong \O(a_1) \oplus \cdots \oplus \O(a_{n-2})\}.\]
For convenience of notation, we will abbreviate $\O(a_1) \oplus \cdots \oplus \O(a_{n-2})$ by $\O(\vec{a})$. 

The loci $\Sigma_{\vec{a}}$ can be realized as the loci where the members of a family of vector bundles on $\pp^1$ acquire certain splitting types. Let
\[\Phi := \{(p, L, X) \in \pp^n \times \gg(1, n) \times U : p \in L \subset X\}\]
be the universal line over $\Sigma$ and let $\pi: \Phi \rightarrow \Sigma$ be the projection map. In addition, let
\[\Psi := \{(p, L, X) \in \pp^n \times \gg(1, n) \times U : p \in X \text{ and } L \subset X\}\]
be the universal hypersurface over $\Sigma$. Then, the vector bundle $\mathcal{N}:= N_{\Phi/\Psi}$ has the property that for any $(L, X) \in \Sigma$, the restriction $\N|_{\pi^{-1}(L, X)}$ is $N_{L/X}$. It follows that the $\Sigma_{\vec{a}}$ form a stratification of $\Sigma$ with
\[\overline{\Sigma_{\vec{a}}} = \bigcup_{\vec{a}' \leq \vec{a}} \Sigma_{\vec{a}}\]
where the partial ordering $\leq$ is defined by
\[\vec{a}' \leq \vec{a} \qquad \Longleftrightarrow \qquad a_1' + \ldots + a_k' \leq a_1 + \ldots + a_k \quad \text{for all $k$}.\]
See for example Section 14.4.1 of \cite{EH}. There is a unique maximal element with respect to this partial ordering, which is determined by the condition $|a_i - a_j| \leq 1$ for all $i$ and $j$. We call this the \textit{balanced} splitting type and all others \textit{unbalanced}.

In this scenario of a family of vector bundles on $\pp^1$, deformation theory gives rise to an expected codimension for the loci $\Sigma_{\vec{a}}$.
Suppose $\mathcal{E}$ is any family of vector bundles on $\pp^1$ with base $B$, and let $\pi: \pp^1 \times B \rightarrow B$ be the projection map. For each point $b \in B$ there is an analytic neighborhood $B'$ of $b$ and a map from $B'$ to the deformation space of the vector bundle $\mathcal{E}|_{\pi^{-1}(b)}$ such that $\E|_{\pi^{-1}(B')}$ is equal to the pullback of the versal family on the deformation space.
 The codimension of the locus of points $b \in B$ where $\E|_{\pi^{-1}(b)} \cong \O(\vec{a})$ therefore has codimension at most the dimension of the deformation space of $\O(\vec{a})$.
We call this quantity the expected codimension for the locus where members acquire splitting type $\vec{a}$ and denote it by
\[u(\vec{a}) := h^1(\mathcal{E}nd(\O(\vec{a}))) = \sum_{i < j} \max\{a_j - a_i - 1, 0\}.\]

Our main results are the following.

\begin{Theorem} \label{1}
$\Sigma_{\vec{a}}$ is smooth and irreducible of codimension $u(\vec{a})$ in $\Sigma$.
\end{Theorem}

\begin{remark}
Although the open strata $\Sigma_{\vec{a}}$ are smooth, their closures can and will be singular along the more unbalanced splitting types.
\end{remark}

\begin{Example}
For degree $7$ hypersurfaces in $\pp^5$, the following diagram indicates which strata lie in the closure of others and the codimension of each strata in $\Sigma$.
Here, the balanced splitting type is $(-1, -1, -1)$ at the far right.

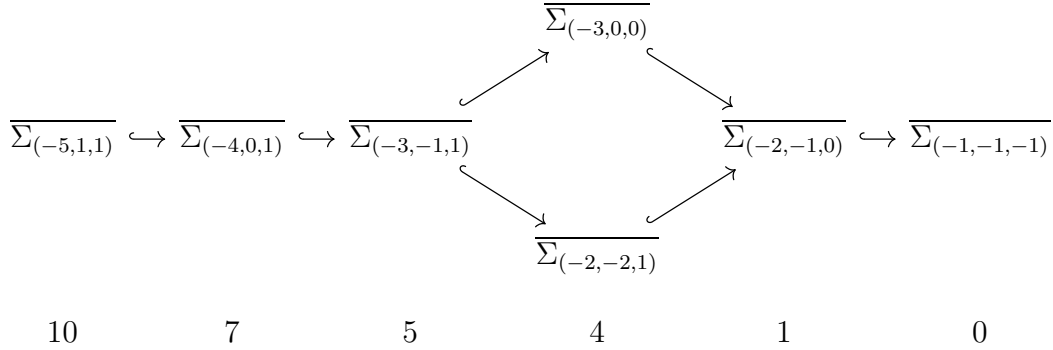
\begin{center}
\hspace{-.25in}
\begin{tikzcd}[column sep =small] 
&&&&\overline{\Sigma_{(-3, 0, 0)}} \arrow[hookrightarrow]{rd}\\
&\overline{\Sigma_{(-5,1,1)}} \arrow[hookrightarrow]{r} & \overline{\Sigma_{(-4, 0, 1)}} \arrow[hookrightarrow]{r} & \overline{\Sigma_{(-3, -1, 1)}} \arrow[hookrightarrow]{ru}  \arrow[hookrightarrow]{rd} &&\overline{\Sigma_{(-2,-1,0)}} \arrow[hookrightarrow]{r} & \overline{\Sigma_{(-1, -1, -1)}}\\
&&&&\overline{\Sigma_{(-2, -2, 1)}} \arrow[hookrightarrow]{ru} \\[-12pt]
&10 & 7 & 5 & 4 & 1 & 0
\end{tikzcd}
\end{center}

\vspace{.1in}
\noindent
Note that the splitting types $(-3, 0, 0)$ and $(-2, -2, 1)$ cannot specialize to each other, showing that in general the splitting types are not totally ordered.
\end{Example}

\begin{Theorem} \label{2}
If $u(\vec{a}) > 2n - d - 3$ then $F_{\vec{a}}(X)$ is empty for general $X$. If $u(\vec{a}) \leq 2n - d - 3$, then $F_{\vec{a}}(X)$
has codimension $u(\vec{a})$ inside $F(X)$ for general $X$. In this case, the class of $\overline{F}_{\vec{a}}(X)$ in the Chow ring of $\gg(1, n)$ is computed by the formula in Proposition \ref{classp}.
\end{Theorem}

\begin{Theorem} \label{3}
If $d \geq 3$, then $\dim F_{(-1,1,\ldots,1)}(X) \leq n-3$ for all smooth hypersurfaces $X \subset \pp^n$ of degree $d$. If $d \geq 4$, then $\dim F_{(-a,-b,1,\ldots,1)}(X) \leq n-1$ for all smooth hypersurfaces $X \subset \pp^n$ of degree $d$.
\end{Theorem}

\begin{remark}
The questions in this paper could just as well be asked for higher degree rational curves. Recent work of Riedl and Yang \cite{RY} shows that when $n \geq d+2$, the locus of rational curves of degree $e$ on general hypersurfaces $X \subset \pp^n$ of degree $d$ has the ``expected codimension" $e(n-d+1)+n-4$. However, the results of Coskun and Riedl in \cite{CR} on normal bundles of rational curves in $\pp^n$ suggest that the normal bundles of higher degree rational curves on hypersurfaces may be less well behaved.
\end{remark}

This paper is organized as follows.
In the next section, we put a scheme structure on $\Sigma_{\vec{a}}$ and explain how to compute the class of $\overline{\Sigma_{\vec{a}}}$ for certain splitting types $\vec{a}$, assuming they have the expected codimension. We also provide explicit local equations for $\Sigma_{\vec{a}}$ and describe its functor of points. In Section 3, we prove Theorem \ref{1}. We then describe the tangent space to $F_{\vec{a}}(X)$ and prove Theorem \ref{2} in Section 4. In Section 5, we study singularities on $F_{\vec{a}}(X)$ when $X$ is a cubic threefold and give an important example of the scheme structure of $F_{\vec{a}}(X)$ for $X$ the Fermat quartic threefold.
In Section 6, we find the class of $\overline{F_{\vec{a}}(X)}$ and compute the number of lines with unbalanced normal bundle on a general quintic fourfold as an example. Finally, in Section 7, we prove Theorem \ref{3}.

\subsection*{Acknowledgements.}
First and foremost, I would like to thank Professor Joe Harris for all of his encouragement and advice, and meeting with me weekly to discuss ideas. I am also grateful to the 2017 Harvard Program for Research in Science and Engineering (PRISE) and the Herchel Smith Fellowship for their support last summer, when I began working on this project. Finally, thanks to James Hotchkiss for many helpful conversations about this topic and algebraic geometry in general.

\section{The scheme $\Sigma_{\vec{a}}$}
Here, we give $\Sigma_{\vec{a}}$ the structure of a scheme by realizing it as an intersection of loci where certain maps of vector bundles drop rank.
We work on the closure of the universal Fano scheme
\[\overline{\Sigma} := \{(L, X) \in \gg(1, n) \times \pp^N : L \subset X\},\]
and introduce incidence correspondences
\[\overline{\Phi} := \{(p, L, X) \in \pp^n \times \gg(1, n) \times \pp^N : p \in L \subset X\},\]
and
\[\overline{\Psi} := \{(p, L, X) \in \pp^n \times \gg(1, n) \times \pp^N : p \in X \text{ and } L \subset X\},\]
which are just the closures in $\pp^n \times \gg(1, n) \times \pp^N$ of the varieties defined earlier with the same letters.
Next, we set
 \[\N := N_{\overline{\Phi}/\overline{\Psi}} \qquad  \E := N_{\overline{\Phi}/ \pp^n \times \overline{\Sigma}} \qquad \text{and} \qquad \cF := N_{\overline{\Psi}/\pp^n \times \overline{\Sigma}}|_{\overline{\Phi}}\]
 so that we have a short exact sequence of sheaves on $\overline{\Phi}$,
 \begin{equation} \label{seq}
 \begin{tikzcd}
 &0 \arrow{r} &\N \arrow{r}  &\E \arrow{r}{\phi} &\cF \arrow{r}  &0.
 \end{tikzcd}
 \end{equation}
 Note that $\E$ and $\cF$ are vector bundles, but since $\overline{\Psi}$ is singular, $\N$ is not. Let $\mathcal{Q}$ be the universal quotient bundle on $\gg(1, n)$.
Labeling the relevant projections
\begin{equation}
\begin{tikzcd}[column sep = small]
&&\overline{\Phi} \arrow{dl}[swap]{\alpha} \arrow{dr}{\pi}\\
&\pp^n & & \overline{\Sigma} \arrow{dl}[swap]{\rho} \arrow{dr}{\beta}\\
&&\gg(1,n) & & \pp^N,
\end{tikzcd}
\end{equation}
one readily identifies the vector bundles $\E$ and $\cF$ as
\[
\E = (\rho \circ \pi)^* \mathcal{Q} \otimes \alpha^* \O_{\pp^n}(1) \qquad 
\text{and} \qquad
 \cF = \alpha^* \O_{\pp^n}(d) \otimes (\beta \circ \pi)^* \O_{\pp^N}(1).\]
 
For each $i \geq -1$, tensoring \eqref{seq} with $\alpha^*\O_{\pp^n}(i)$ gives rise to a new short exact sequence
\begin{equation} \label{seq2}
 \begin{tikzcd}
 &0 \arrow{r} &\N(i) \arrow{r}  &\E(i) \arrow{r}{\phi(i)} &\cF(i) \arrow{r}  &0.
 \end{tikzcd}
\end{equation}
Applying pushforward by $\pi$ to the above sequence \eqref{seq2}, we obtain an exact sequence of sheaves on $\overline{\Sigma}$,
\begin{equation} \label{seq3}
 \begin{tikzcd}
 &0 \arrow{r} &\pi_*\N(i) \arrow{r}  &\pi_*\E(i) \arrow{r}{\pi_*\phi(i)} &\pi_*\cF(i) \arrow{r} & R^1\pi_* \N(i) \arrow{r} & 0.
 \end{tikzcd}
\end{equation}
Since
\begin{equation} \label{rankE}
h^0(\E(i)|_{\pi^{-1}(L, X)}) = h^0(\O_L(1+i)^{n-1}) = (i+2)(n-1)
\end{equation}
and
\begin{equation} \label{rankF}
h^0(\cF(i)|_{\pi^{-1}(L,X)}) = h^0(\O_L(d+i)) = d + i + 1
\end{equation}
are constant as $(L, X)$ vary over $\overline{\Sigma}$, the theorem on cohomology and base change tells us that $\pi_*\E(i)$ and $\pi_*\cF(i)$ are vector bundles on $\overline{\Sigma}$. Indeed, using the push-pull formula for vector bundles and denoting by $\mathcal{S}^*$ the dual of the universal subbundle on $\gg(1, n)$, we identify
\[\pi_*\E(i) = \rho^* \mathcal{Q} \otimes \rho^* \mathrm{Sym}^{1+i} \mathcal{S}^* \qquad \text{and} \qquad \pi_*\cF(i) = \rho^* \mathrm{Sym}^{d+i} \mathcal{S}^* \otimes \beta^* \O_{\pp^N}(1).\]

Given any map of vector bundles $\psi: E \rightarrow F$ on some base $B$, let $M_{k}(\psi)$ denote the subscheme of $B$ where $\psi$ has rank at most $k$, defined by the $(k+1) \times (k+1)$ minors of $\psi$. Furthermore, let $M'_{k}(\psi)$ be the locally closed subscheme $M_{k}(\psi) \backslash M_{k-1}(\psi)$ where $\psi$ has rank exactly $k$.
For each $k \geq 0$ and $i \geq -1$, looking at \eqref{seq2}, we see that the locus $M_k'(\pi_*\phi(i))$ where $\pi_*\phi(i)$ has rank exactly $k$ is
\[M_k'(\pi_*\phi(i)) = \{(L, X) \in \overline{\Sigma}: h^0(\N(i)_{(L,X)}) = (i+2)(n-1) - k\}.\]
Meanwhile, given a splitting type $\vec{a} = (a_1, \ldots, a_{n-2})$, the locus $\Sigma_{\vec{a}}$ can be described as
\[\Sigma_{\vec{a}} = \{(L, X) \in \Sigma: h^0(\N(i)_{(L,X)}) = t_i(\vec{a}) \text{ for all $i \geq -1$}\},\]
where
\[t_i(\vec{a}) := h^0(\O(\vec{a}) \otimes \O(i)) = \sum_{j=1}^{n-2} \max\{0, a_j + i + 1\}. \]
Thus, we can give $\Sigma_{\vec{a}}$ the structure of a scheme by taking an appropriate intersection of schemes $M_{k}'(\pi_*\phi(i))$. For convenience, let
\[r_i(\vec{a}) := (i+2)(n-1) - t_i(\vec{a}).\]
\begin{definition}
We define the scheme $\Sigma_{\vec{a}}$ as
\begin{equation}\label{defSigma}
\Sigma_{\vec{a}} = \Sigma \cap \bigcap_{i=-1}^{d-2} M_{r_i(\vec{a})}'(\pi_*\phi(i)).
\end{equation}
The loci $F_{\vec{a}}(X)$ then inherit a scheme structure as the fibers of $\Sigma_{\vec{a}}$ under projection to the $\pp^N$ parameterizing hypersurfaces of our given degree and dimension.
\end{definition}

\begin{Example}[Cubic hypersurfaces]
There are only two splitting types for the normal bundle of a line on a cubic hypersurface: either
\[N_{L/X} = \O(-1) \oplus \O(1)^{n-3} \qquad \text{or} \qquad N_{L/X} = \O \oplus \O \oplus \O(1)^{n-4}.\]
Let $\vec{a} = (-1, 1, \ldots, 1)$ denote the unbalanced splitting type and $\vec{b} = (0, 0, 1, \ldots, 1)$ the balanced splitting type. We have
\begin{align*}
t_i(\vec{a}) &= \begin{cases} n-3 & \text{if $i=-1$} \\ i+(i+2)(n-3) & \text{if $i \geq 0$} \end{cases} \quad &\text{and} \qquad
 t_i(\vec{b}) &= i+(i+2)(n-3) \quad \forall i,
\intertext{and hence}
r_i(\vec{a}) &= \begin{cases} 2 & \text{if $i=-1$} \\ i+4 & \text{if $i \geq 0$} \end{cases} \quad &\text{and} \qquad
 r_i(\vec{b}) &= i+4 \quad \forall i.
 \end{align*}
Since $r_i(\vec{a}) = r_i(\vec{b})$ for all $i \geq 0$, the locus $M_{r_i(\vec{a})}'(\pi_*\phi(i))$ is all of $\Sigma$ for these $i$. Therefore, the only non-trivial term in the intersection \eqref{defSigma} comes from when $i = -1$, giving
\[\Sigma_{\vec{a}} = \Sigma \cap M_2(\pi_*\phi(-1)).\]
\end{Example}

\subsection{The class of certain $\overline{\Sigma_{\vec{a}}}$} \label{clsec}

As seen in the previous example, sometimes $\overline{\Sigma_{\vec{a}}}$ is equal to $M_k(\pi_*\phi(i))$ for some $k$ and $i$. In this case, assuming $\overline{\Sigma_{\vec{a}}}$ has the correct codimension, Porteous' formula will give rise to a formula for the class of $\overline{\Sigma_{\vec{a}}}$ in the Chow ring of $\overline{\Sigma}$.

To set up this formula we need some notation. Given an element $\gamma$ of the Chow ring of a projective scheme $B$, let $\gamma_i$ be the component of $\gamma$ in degree $i$. Then, given any $\gamma$ and natural numbers $e$ and $f$, let 
\begin{equation} \label{ddelta}
\Delta^e_f(\gamma) = \left(\begin{matrix}\gamma_f & \gamma_{f+1}  & \cdots & \gamma_{e+f-1} \\[8pt]
\gamma_{f-1} & \gamma_f & \cdots & \gamma_{e+f-2} \\[10pt]
\vdots & \vdots & \ddots & \vdots \\[10pt]
\gamma_{f-e+1} & \gamma_{f-e+2} & \cdots & \gamma_{f}\end{matrix}\right).
\end{equation}
In general, Porteous' formula says that if $\psi: E \rightarrow F$ is any map of vector bundles of ranks $e$ and $f$ on $B$ and the scheme $M_k(\psi)$ has codimension $(f-k)(e-k)$ in $B$, then the class of $M_k(\phi)$ in the Chow ring of $B$ is
\[[M_k(\phi)] = \Delta^{e-k}_{f-k}\left(\frac{c(F)}{c(E)}\right),\]
where $c(F)$ and $c(E)$ are the total Chern classes of $F$ and $E$.

\vspace{.1in}
The $\vec{a}$ for which we will calculate the class of $\overline{\Sigma_{\vec{a}}}$ have the form 
\begin{equation} \label{good}
\vec{a} = (s_1, \ldots, s_{n-2-m}, \underbrace{1, \ldots, 1}_m) \qquad \text{where} \qquad \sum_{i < j} \max\{0, s_j - s_i - 1\} = 0.
\end{equation}
The closure of such $\Sigma_{\vec{a}}$ in $\Sigma$ consists of pairs $(L, X)$ where $N_{L/X}$ has at least $m$ copies of $\O(1)$ in it, whose closure in $\overline{\Sigma}$ is exactly $M_{r_{-1}(\vec{a})}(\pi_*\phi(-1)).$
In addition, for such $\vec{a}$, we always have
\[u(\vec{a}) = mb =(\mathrm{rank}(\pi_*\E(-1)) - r_{-1}(\vec{a}))(\mathrm{rank}(\pi_*\cF(-1)) - r_{-1}(\vec{a}))\]
where $b = m-(n-d-1)$. Applying Porteous' formula, we immediately arrive at the following.

\begin{Proposition} \label{classes}
Suppose $\vec{a}$ has the form \eqref{good}, and set $b = m-(n-d-1)$. If $\overline{\Sigma_{\vec{a}}}$ has the expected codimension $u(\vec{a})$ in $\overline{\Sigma}$, then its class in the Chow ring of $\overline{\Sigma}$ is given by
\[\left[ \ \overline{\Sigma_{\vec{a}}} \ \right] =  \Delta^{m}_{b}\left(\frac{c(\pi_*\cF(-1))}{c(\pi_*\E(-1))}\right).\]
\end{Proposition}

\begin{remark}
Given that each $\Sigma_{\vec{a}}$ has the ``expected codimension", it is natural to wonder if a similar formula can be found for all $\vec{a}$. One obstacle to following the same approach as above is that for $i > -1$, the loci $M_{r_i(\vec{a})}(\pi_*\phi(i))$ pick up components of larger dimension. For example, when $n = 6$ and $d = 11$, so that the rank of the normal bundle is $4$ and its degree is $-6$, we have
\[M_{2}(\pi_*\phi(0)) = \overline{\Sigma_{(-3,-3,0,0)}} \cup \overline{\Sigma_{(-3,-2,-2,1)}}.\]
\end{remark}

To use Proposition \ref{classes} in practice, we need to describe the Chow ring of $\overline{\Sigma}$ and find the Chern classes of $\pi_*\E(i)$ and $\pi_*\cF(i)$. This is easily done since $\overline{\Sigma}$ is a projective bundle over $\gg(1, n)$ and the vector bundles $\pi_*\E(i)$ and $\pi_*\cF(i)$ are tensor products of vector bundles with known Chern classes, and hence their
 Chern classes are determined by the splitting principle.

\subsection{Explicit local equations} 

For each line $L _0 \in \gg(1, n)$, we can choose coordinates on $\pp^n$ so that $L _0 = V(x_2, \ldots, x_n)$ and affine coordinates $a_{ij}$ on $\gg(1, n)$ where
 \[\left(\begin{matrix} a_{20} & a_{30} & \cdots & a_{n0} \\ a_{21} & a_{31} & \cdots & a_{n1} \end{matrix}\right) \quad \longleftrightarrow \quad L: x_i = a_{i0} x_0 + a_{i1} x_1.\]
Over the product of this affine open with $\pp^N$, the map $\phi$ in \eqref{seq} is given by $n-1$ homogeneous polynomials of degree $d-1$,
\[\frac{\partial f}{\partial x_i}(x_0, x_1, a_{20}x_0 + a_{21}x_1, \ldots, a_{n0} x_0 + a_{n1} x_1) = C_{i0} x_0^{d-1} + \ldots + C_{id-1} x_1^{d-1} \]
where the $C_{ij}$ are bihomogeneous polynomials in the $a_{k\ell}$ and the coefficients of $f$, which are our coordinates on $\pp^N$.

Over this open subset, the vector bundle $\pi_* \E(i)$ is naturally identified with the trivial bundle with fiber $H^0(\O_{L_0}(1+i))^{n-1}$ and similarly $\pi_*\cF(i)$ is naturally identified with $H^0(\O_{L_0}(d+i)) \otimes \beta^*\O_{\pp^N}(1)$.
In terms of the standard basis of monomials on each copy of $H^0(\O(1+i))$ and $H^0(\O(d+i))$, the map $\pi_*\phi(i)$ is represented by the matrix
\[C(i) = \left(
\begin{matrix}
C_{2,0} & 0  & \cdots  & 0 &&\hdots\\
C_{2,1} & C_{2, 0} & \cdots & 0 &&\hdots\\
\vdots & C_{2,1}  & \ddots & \vdots &&\hdots\\
C_{2,d-1} & \vdots & \ddots & C_{2,0} &&\hdots\\
0 & C_{2, d-1} &  & C_{2,1} &&\hdots\\
\vdots & \vdots & \ddots & \vdots&&\hdots \\
0 & 0 & \cdots & C_{2,d-1}&&\hdots
\end{matrix} 
\begin{matrix}
&&C_{n,0} & 0  & \cdots  & 0\\
&&C_{n,1} & C_{n, 0} & \cdots & 0\\
&&\vdots & C_{n,1}  & \ddots & \vdots\\
&&C_{n,d-1} & \vdots & \ddots & C_{n,0} \\
&&0 & C_{n, d-1} &  & C_{n,1}\\
&&\vdots & \vdots & \ddots & \vdots \\
&& 0 & 0 & \cdots & C_{n,d-1}
\end{matrix}\right)
\]
The local equations for $M_{r_i(\vec{a})}(\pi_*\phi(i))$ are exactly the $(r_i(\vec{a})+1) \times (r_i(\vec{a})+1)$ minors of this matrix. The collection of all these minors as $i$ runs from $-1$ to $d-2$ are thus local equations for $\overline{\Sigma_{\vec{a}}}$.

\begin{Example}[$n=4,d=3$] \label{cubie4}
We have $\Sigma_{\vec{a}} = M_2(\pi_*\phi(-1))$, which is defined by the single equation
\[\det C(-1) = \det \left(\begin{matrix} C_{2,0} & C_{3,0} & C_{4,0} \\ C_{2,1} & C_{3,1} & C_{4,1} \\ C_{2,2} & C_{3,2} & C_{4,2}\end{matrix}\right).\]
Since this equation is non-zero, this provides a direct proof that $\Sigma_{\vec{a}} \subset \Sigma$ has codimension $1$, which is the expected codimension for this locus.
\end{Example}

\subsection{The functor of points}
As one might hope, $\Sigma_{\vec{a}}$ has a nice description in terms of its functor of points. 
Given any morphism $\eta: Z \rightarrow \Sigma$, one obtains a diagram
\begin{equation} \label{deformation}
\begin{tikzcd}
&\tilde{\eta}^*\N \arrow{r} \arrow{d} &\N \arrow{d} \\
&\eta^*\Phi \arrow{d}\arrow{r}{\tilde{\eta}} & \Phi \arrow{d} \\
&Z \arrow{r}{\eta} & \Sigma,
\end{tikzcd} 
\end{equation}
where $\tilde{\eta}^*\N$ is a family of vector bundles on $\pp^1$ with base $Z$. We say that such a family has \textit{constant local splitting type $\vec{a}$} if $Z$ can be covered by open sets $U$ over which $(\eta^*\Phi)_U \cong U \times \pp^1$ and
\[(\tilde{\eta}^*\N)_{U} \cong \alpha^*\O(\vec{a})\]
where $\alpha: (\eta^*\Phi)_U \cong U \times \pp^1 \rightarrow \pp^1$ is the projection map. With this notion, we have the following.

\begin{Lemma} Let $\N = N_{\Phi/\Psi}$ as before. Then
\[\mathrm{Mor}(Z, \Sigma_{\vec{a}}) = \{\eta \in \mathrm{Mor}(Z, \Sigma): \tilde{\eta}^*\N \text{ has constant local splitting type $\vec{a}$}\}.\]
\end{Lemma}

\begin{proof}
First we check that the restriction of $\N$ to $\Sigma_{\vec{a}}$ has constant local splitting type $\vec{a}$. For each $(L, X) \in \Sigma_{\vec{a}}$, let $U \subset \Sigma_{\vec{a}}$ be an open set containing $(L, X)$ over which $\E_U$ can be trivialized, i.e.
\[\E_U \cong H^0(\O_L(1)^{n-1}) \otimes \O_U.\]
By the theorem on cohomology and base change, $\pi_*\N(i)$ is a vector bundle on $\Sigma_{\vec{a}}$. After passing to a possibly smaller open set $U$, we can assume that
\[\pi_*\N(i)_U = \ker \pi_*\phi(i)_U: \pi_*\E(i)_U \rightarrow \pi_*\cF(i)_U\]
can be trivialized for all $i = -1, \ldots, d-2$.
An inclusion of $\alpha^*\O(a_k)$ into the vector bundle $\E_U \cong H^0(\O_L(1)^{n-1}) \otimes \O_U$ is specified by a collection of $n-1$ polynomials of degree $1-a_k$ in $x_0, x_1$ with coefficients rational functions on $U$ that do not simultaneously vanish, i.e. a non-zero section of 
\[H^0(\O_L(1-a_k)^{n-1}) \otimes \O_U = \pi_*\E(-a_k)_U.\]
 Such a section will define an inclusion of $\alpha^*\O(a_k)$ into $\N_{U}$ exactly when the section lies in the kernel of $\pi_*\phi(-a_k)_U$.

Now let $m_i$ be the number of $i$'s in the list $\vec{a}$. Suppose we have a local frame of the kernel of $\pi_*\phi(-1)_U$,
\[w^{(k)} = (w_2^{(k)}, \ldots, w_{n}^{(k)}) \in H^0(\O_L^{n-1}) \otimes \O_U = \O_U^{n-1} \qquad \qquad k=1, \ldots, m_1.\]
We observe that the elements 
\[(w_2^{(k)} x_0, \ldots, w_n^{(k)}x_0) \text{ and } (w_2^{(k)}x_1, \ldots, w_n^{(k)}x_1) \in H^0(\O(1))^{n-1} \otimes \O_U \]
are in the kernel of $\pi_*\phi_U$. Since $\ker \pi_*\phi_U = \pi_*\N_U$ has rank $2m_1 + m_0$, we can find another $m_0$ independent sections of $H^0(\O(1)^{n-1}) \otimes \O_U$ in the kernel of $\pi_*\phi_U$. These $m_0$ sections define an inclusion $\alpha^*\O^{m_0} \rightarrow \N$ which is independent from our previous inclusion of $\alpha^*\O(1)^{m_1}$.

Continuing in this way, the conditions on the rank of $\pi_*\phi(i)$ guarantee that we can find sections in its kernel that define inclusions of $\alpha^*\O(-i)^{m_{-i}}$ independent from those defined before for lower $i$. Hence, the restriction of $\N$ to $\Sigma_{\vec{a}}$ has constant local splitting type $\vec{a}$.

It remains to show that if $\eta: Z \rightarrow \Sigma$ is such that $\tilde{\eta}^*\N$ has constant local splitting type $\vec{a}$ then $\eta$ factors through $\Sigma_{\vec{a}}$. If the restriction of $\tilde{\eta}^*\N$ to some $U \subset Z$ is $\alpha^*\O(\vec{a})$, then the $(r_i(\vec{a})+1) \times (r_i(\vec{a})+1)$ minors of $\tilde{\eta}^*\phi(i)$ vanish. But this is given locally by $\eta^\#$ applied to the minors of $C(i)$, where $\eta^\#: \O_{\Sigma} \rightarrow \O_Z$ is the map of structure sheaves. Thus, $\eta^\#$ kills the ideal of $\Sigma_{\vec{a}}$, which is to say $\eta$ factors through the subscheme $\Sigma_{\vec{a}} \subset \Sigma$.
\end{proof}

\section{Proof of Theorem \ref{1}}
In this section, we prove Theorem \ref{1}, providing two proofs of the dimension statement. The first uses a standard dimension-counting argument with an incidence correspondence and also proves smoothness and irreducibility. The second uses deformation theory to obtain the dimension statement directly. Recall the statement of the theorem.
\begin{Theorem} \label{uni-ecd}
 $\Sigma_{\vec{a}}$ is smooth and irreducible of codimension $u(\vec{a})$ in $\Sigma$.
\end{Theorem}
\begin{proof}
First note that, like the universal Fano scheme, the projection onto the second factor $\rho_{\vec{a}}: \Sigma_{\vec{a}} \rightarrow \gg(1, n)$ is a fiber bundle. Thus, to prove the theorem, it will suffice to prove that the fibers of $\rho_{\vec{a}}: \Sigma_{\vec{a}} \rightarrow \gg(1, n)$ are smooth and irreducible of codimension $u(\vec{a})$ in the fibers of the projection of the universal Fano scheme onto its second factor, $\rho: \Sigma \rightarrow \gg(1, n)$. 

So, fix a line $L = V(x_2, \ldots, x_n)$ and let
\[\Omega = \{(X, \varphi) \in U \times \mathrm{Aut}(\O(\vec{a})) : X \supset L \text{ and } \varphi: \O(\vec{a}) \cong N_{L/X}\},\]
where $U \subset \pp^N$ is the open subset parameterizing smooth hypersurfaces.
First consider the projection $\alpha: \Omega \rightarrow U$. The image of $\alpha$ is exactly $\rho_{\vec{a}}^{-1}(L)$, and $\alpha$ makes $\Omega$ into a fiber bundle over its image with smooth fibers isomorphic to $\mathrm{Aut}(\O(\vec{a}))$. Irreducibility and smoothness of $\rho_{\vec{a}}^{-1}(L)$ will thus follow from the same properties of $\Omega$.

Next, we construct a map $\beta: \Omega \rightarrow \mathrm{Hom}(\O(\vec{a}), \O(1)^{n-1})$ by setting $\beta((X, \varphi))$ to be the composition
\[\varphi: \O(\vec{a}) \cong N_{L/X} \longrightarrow N_{L/X} \cong \O(1)^{n-1}.\]
The image of $\beta$ is certainly contained in the open subset of injective maps, and we claim that all injective maps are in the image. Indeed, given any
\begin{center}
\begin{tikzcd}
&0 \arrow{r} & \O(\vec{a}) \arrow{r}{A} & \O(1)^{n-1},
\end{tikzcd}
\end{center}
the cokernel is a line bundle of degree $d$, so we have a short exact sequence
\begin{equation} \label{nor1}
0 \longrightarrow \O(\vec{a}) \xrightarrow{A} \O(1)^{n-1} \xrightarrow{(m_1, \ldots, m_{n-1})} \O(d) \longrightarrow 0.
\end{equation}
Explicitly, if we think of $A$ as being represented by an $(n-1) \times (n-2)$ matrix with polynomials of degree $1 - a_j$ in the $j$th column, then the $m_i$ are the maximal minors of $A$. We claim that
\begin{equation} \label{dee}
\beta^{-1}(A) = \left\{(V(f), \varphi): f = \sum_{i=2}^{n} x_i m_{i+1} + (x_2, \ldots, x_n)^2 \text{ smooth}\right\}.
\end{equation}
If $X = V(f)$ is any hypersurface containing $L$, then we have a short exact sequence
\begin{equation} \label{nor2}
0 \rightarrow N_{L/X} \xrightarrow{\iota} \O_L(1)^{n-1} \xrightarrow{(f_2, \ldots, f_n)} \O_L(d) \rightarrow 0,
\end{equation}
where $f_i = \frac{\partial f}{\partial x_i}|_L$. So for $(X, \varphi)$ to be in $\beta^{-1}(A)$, we must have $f_i = m_{i+1}$, and hence $f$ has the claimed form. On the other hand, whenever $f$ has this form, the right hand maps in \eqref{nor1} and \eqref{nor2} are equal, so by the universal property of kernel there exists a unique isomorphism of their kernels $\varphi: \O(\vec{a}) \cong N_{L/X}$ such that $\varphi \circ \iota = A$, implying $\beta((V(f), \varphi)) = A$.

Next, observe that the $m_i$ have no common zeros, so every $f$ of the form in \eqref{dee} is smooth along $L$. Thus, Bertini's theorem tells us that the general such $f$ is smooth. That is, $\beta^{-1}(A)$ is an open dense subset of this linear system. In particular, $\beta^{-1}(A)$ is smooth and irreducible of dimension
\[\dim\beta^{-1}(A) = (N+1) - d - 1 - (n-1)d .\]
Our assumption that $a_i \leq 1$ for all $i$ guarantees that the image $\beta(\Omega)$, which is equal to the injective maps, is an open dense subset of $H^0(\mathcal{H}om(\O(\vec{a}), \O(1)^{n-1}))$. It follows that $\Omega$ is smooth and irreducible of dimension
\[\dim \Omega = h^0(\mathcal{H}om(\O(\vec{a}), \O(1))) + (N+1) - d - 1 - (n-1)d.\]
Finally, since the fibers of $\alpha$ are copies of $\mathrm{Aut}(\O(\vec{a}))$, we have
\[\dim \alpha(\Omega) = h^0(\mathcal{H}om(\O(\vec{a}), \O(1))) + (N+1) - d - 1 - (n-1)d - h^0(\mathcal{E}nd(\O(\vec{a}))).\]
Hence,
\begin{align*}
\mathrm{codim}(\rho_{\vec{a}}^{-1}(L) \subset \rho^{-1}(L)) &= N-d-1 - \dim \alpha(\Omega) \\ 
&= d(n-1) - 1 - h^0(\mathcal{H}om(\O(\vec{a}), \O_L(1)^{n-1})) \\
&\qquad \qquad \qquad \quad + h^0(\mathcal{E}nd(\O(\vec{a}))).
\end{align*}
Next, observe that
\begin{align*}
h^0(\mathcal{H}om(\O(\vec{a}), \O(d))) &= \sum_{i=1}^{n-2} h^0(\O(d - a_i)) = \sum_{i=1}^{n-2} (d - a_i  +1) \\
&= (d+1)(n-2) - (n-d-1) \\
&= d(n-1) -1.
\end{align*}
After substituting this in above, applying $\mathcal{H}om(\O(\vec{a}), -)$ to the short exact sequence
\begin{center}
\begin{tikzcd}
&0 \arrow{r} &\O(\vec{a}) \arrow{r} &\O_L(1)^{n-1} \arrow{r} &\O(d)\arrow{r} &0,
\end{tikzcd}
\end{center}
and setting the alternating sum of the dimension of terms in the long exact sequence in cohomology to zero, we see that 
\[\mathrm{codim}(\rho_{\vec{a}}^{-1}(L) \subset \rho^{-1}(L)) = h^1(\mathcal{E}nd(\O(\vec{a}))) = u(\vec{a}),\]
as desired.
\end{proof}

Given that $\Sigma_{\vec{a}}$ has the expected codimension $h^1(\mathcal{E}nd(\O(\vec{a})))$ coming from deformation theory, one might wonder if there is a direct proof of this fact using deformation theory. This can be carried out as follows.

\begin{proof}[Alternative proof of dimension count]
For each point $(L, X) \in \Sigma$, there exists an analytic neighborhood $B$ of $(L, X)$ and a map $\phi: B \rightarrow \Delta = \mathrm{Def}(N_{L/X})$ so that 
\[\N|_{\pi^{-1}(B)} \cong \phi^* \cF\]
where $\cF \rightarrow \Delta \times \pp^1$ is the miniversal family. The claim will follow from showing that the differential of $\phi$ is surjective at $(L, X)$. To do this, we construct a set $B' \subset B$ containing $(L, X)$ which lifts an open neighborhood of  the distinguished point $0 \in \Delta$.

Consider the projections
\begin{center}
\begin{tikzcd}
&&\cF \arrow{d} & \\
&&\Delta \times \pp^1 \arrow{dl}[swap]{\alpha} \arrow{dr}{\beta}\\
&\Delta  && \pp^1.
\end{tikzcd}
\end{center}
First, we will show that there exists a neighborhood $V$ of $0 \in \Delta$ so that $\cF|_{\alpha^{-1}(V)}$ admits an inclusion into the trivial family $\beta^*\O(1)^{n-1}$. This is equivalent to finding a local section of $\alpha_*\mathcal{H}om(\cF, \beta^* \O(1)^{n-1})$ which is corresponds to an injective map on each fiber. Since each subbline bundle of $\cF_\delta$ for $\delta \in \Delta$ has degree at most $1$, we have 
\[h^0(\cF_\delta, \beta^*\O(1)^{n-1}) = (n-1)\sum_{i=1}^{n-2} (2-a_i(\delta)) = (n-1)(2(n-2)-n-d-1) ,\]
which is constant as $\delta$ varies over $\Delta$. Hence, the theorem on cohomology and base change tells us that $\alpha_*\mathcal{H}om(\E, \beta^* \O(1)^{n-1})$ is a vector bundle. Let $\iota$ be the inclusion
\[\iota: N_{L/X} \rightarrow N_{L/\pp^n} \cong \O(1)^{n-1}.\] Since the rank of a family of maps drops on closed subsets, we can find a local section of $\alpha_*\mathcal{H}om(\cF, \beta^* \O(1)^{n-1})$ which is $\iota$ over the origin and injective on each fiber $\cF_\delta$ for $\delta$ in some neighborhood $V$.

This gives rise to a short exact sequence
\[0 \rightarrow \cF|_{\alpha^{-1}(V)} \rightarrow \beta^*\op(1)^{n-1} \rightarrow \mathcal{L} \rightarrow 0\]
where $\L$ is a line bundle that restricts to $\op(d)$ on each fiber of $\alpha$. For each $v \in V$, the map $\beta^*\op(1)^{n-1} \rightarrow \L$ restricted to $\alpha^{-1}(v)$ is given by a collection of $n-1$ homogeneous polynomials $g_2(v; x_0,x_1), \ldots, g_n(v;x_0,x_1)$ of degree $d-1$. Fixing coordinates so that $L = V(x_2, \ldots, x_n),$ if $f$ is the defining equation of $X$, then 
\[g_i(0;x_0, x_1) = \left.\frac{\partial f}{\partial x_i}\right|_L.\] 
Next, let
\[\tilde{f} = f - \sum_{i=2}^n x_i \left.\frac{\partial f}{\partial x_i}\right|_L \in (x_2, \ldots, x_n)^2\]
and consider the polynoimals
\[f(v) = \sum_{i=2}^n x_ig_i(v; x_0, x_1) + \tilde{f}.\]
Note that all of the $f(v)$ vanish on $L$ and we have $f(0) = f,$ which is smooth. Since singular hypersurfaces are a closed subset, after restricting to a possibly smaller neighborhood $0 \in V' \subset V$, we will have $f(v)$ smooth for all $v \in V'$. In particular, the collection of hypersurfaces defined by $f(v)$ for $v \in V'$ is a slice around $(L, X)$ that $\phi$ maps one-to-one onto $V'$. 
\end{proof}

\section{The dimension of $F_{\vec{a}}(X)$ in general}
In this section, we determine the dimension of $F_{\vec{a}}(X)$ for general $X$.
It follows immediately from Theorem \ref{1} that if $u(\vec{a}) > 2n - d - 3$, then $F_{\vec{a}}(X)$ is empty for general $X$. On the other hand, if $u(\vec{a}) \leq 2n - d - 3$, then provided $\Sigma_{\vec{a}}$ dominates $U \subset \pp^N$, it follows that $\mathrm{codim}( F_{\vec{a}}(X) \subset F(X)) = u(\vec{a})$ for general $X$. Moreover, by upper semicontinuity of the dimension of fibers of $\Sigma_{\vec{a}} \rightarrow U \subset \pp^N$, to prove this holds, it suffices to find some $L \subset X$ such that $\dim T_LF_{\vec{a}}(X) = 2n - d - 3 - u(\vec{a})$.
In the following, we give an explicit description of the tangent space to $F_{\vec{a}}(X)$, and then exhibit such an $L \subset X$.

Let $P$ be the incidence correspondence
\[P = \{(p, L) \in X \times F(X): p \in L\},\]
and let $T = \mathrm{Spec} \cc[\epsilon]/(\epsilon^2)$.
Every $v: T \rightarrow F(X)$ in $\mathrm{Mor}_L(T, F(X)) = T_LF(X)$ gives rise to a diagram
\begin{equation} \label{deformation}
\begin{tikzcd}
&N_{L/X} \arrow{r} \arrow{d} &\E_{v} \arrow{r} \arrow{d} &\N_{P/X\times F(X)} \arrow{d} \\
&L \arrow{r} \arrow{d} &\mathcal{L}_v \arrow{r} \arrow{d} & P \arrow{d} \\
&\spec \cc \arrow{r} &T \arrow{r}{v} & F(X),
\end{tikzcd} 
\end{equation}
where $\E_v$ is a first order deformation of the vector bundle $N_{L/X}$. Given such an $\L_v \cong T \times \pp^1$, let $\beta: \L_v\rightarrow \pp^1$ be the projection map onto the second factor.

\begin{Lemma}
The tangent vector $v: T \rightarrow F(X)$ in $T_LF(X)$ is in the subspace $T_LF_{\vec{a}}(X)$ if and only if the associated first order deformation $\E_v$ of $N_{L/X}$ is trivial, meaning
\[\E_v \cong \beta^*\O_{\pp^1}(a_1) \oplus \ldots \oplus \beta^*\O_{\pp^1}(a_{n-2}).\]
\end{Lemma}
\begin{proof}
Restricting our description of the functor of points for $\Sigma_{\vec{a}}$ to the fiber of $\Sigma_{\vec{a}} \rightarrow U \subset \pp^N$ over $X$, we see that
\begin{align*}
&\mathrm{Mor}(T, F_{\vec{a}}(X)) \\
&=\{f \in \mathrm{Mor}(T, F(X)) : \tilde{f}^* N_{P/X \times F(X)} \text{ has constant local splitting type $\vec{a}$}\}.
\end{align*}
Since $T$ has only one open set, having constant local splitting type $\vec{a}$ is the same as splitting as above. Thus, $T_LF_{\vec{a}}(X) = \mathrm{Mor}_L(T, F_{\vec{a}}(X))$ is the claimed subspace of $T_LF(X)$.
\end{proof}

Using the inclusion $T_LF(X) = H^0(N_{L/X}) \rightarrow H^0(\O(1)^{n-1}) \cong T_L\gg(1, n)$, 
We can represent every $v \in T_LF(X)$ by a collection of $n-1$ linear forms $v = (v_2, \ldots, v_n)$ where $v_{i}(x_0, x_1) = v_i^{(0)} x_0 + v_i^{(1)} x_1$. Such a collection corresponds to the first order deformation 
\[\L \subset T \times X \text{ defined by } V(x_2 - \epsilon v_2, \ldots, x_n - \epsilon v_n).\]
 Equivalently, if $a_{ij}$ are affine coordinates on the Grassmannian around the line $L= V(x_2, \ldots, x_n)$, i.e.
 \[\left(\begin{matrix} a_{20} & a_{30} & \cdots & a_{n0} \\ a_{21} & a_{31} & \cdots & a_{n1} \end{matrix}\right) \quad \longleftrightarrow \quad L: x_i = a_{i0} x_0 + a_{i1} x_1,\]
  this corresponds to the morphism $T \rightarrow \gg(1, n)$ determined by the map of rings 
  \[\cc[a_{ij}] \rightarrow \cc[\epsilon]/(\epsilon^2), \qquad \text{defined by } a_{ij} \mapsto v_i^{(j)}\epsilon.\]

Given $L \subset X = V(f)$, we have seen that the splitting type of $N_{L/X}$ is determined by the polynomials
\[f_i(x_0, x_1) := \left.\frac{\partial f}{\partial x_i}\right|_L.\]
We will now see that if $L$ has splitting type $\vec{a}$, then the tangent space to $F_{\vec{a}}(X)$ at $L$ is in turn determined by the $f_i(x_0, x_1)$ and the higher derivatives
\[H_{ij}(x_0, x_1) := \left. \frac{\partial^2 f}{\partial x_i \partial x_j}\right|_L.\]
As always, choose coordinates so that $L = V(x_2, \ldots, x_n)$. We can write
\[\frac{\partial f}{\partial x_i} = f_i(x_0, x_1) + \sum_{j=2}^n x_j H_{ij}(x_0, x_1) + (x_2, \ldots, x_n)^2.\]
If $v \in T_LF(X) \subset H^0(\O(1)^{n-1})$, then pulling back the short exact sequence
\begin{center}
\begin{tikzcd}
&0 \arrow{r} &N_{P/ X \times F(X)} \arrow{r} &N_{P/\pp^n \times F(X)} \arrow{r} &N_{X \times F(X)/\pp^n \times F(X)}|_{P} \arrow{r} &0
\end{tikzcd}
\end{center}
to $\mathcal{L}$, we see that $\E_v$ fits into a short exact sequence
\begin{center}
\begin{tikzcd}
&0 \arrow{r} &\E_v \arrow{r} &\beta^*\O(1)^{n-1} \arrow{r} &\beta^*\O(d) \arrow{r} &0
\end{tikzcd}
\end{center}
where the map on the right is given by the collection of polynomials
\[\tilde{f}_i = \frac{\partial f}{\partial x_i}(x_0, x_1, \epsilon v_2, \ldots, \epsilon v_n) = f_i(x_0, x_1) + \epsilon \sum_{j=2}^n v_j H_{ij}(x_0, x_1).\]
The deformation $\E_v$ will be trivial --- i.e. a direct sum of $\beta^* \O(a_i)$ --- if and only if the $\tilde{f}_i$ satisfy the same syzygies as the $f_i(x_0, x_1)$, now with coefficients in $\cc[\epsilon]/(\epsilon^2)$. To make this precise, suppose that the maps in the normal bundle sequence at our line are given explicitly by
\[ 0 \rightarrow N_{L/X} \xrightarrow{(c_{ik})} \O(1)^{n-1} \xrightarrow{(f_2, \ldots, f_n)} \O(d) \rightarrow 0\]
where the $c_{ik}$ are homogenous polynomials of degree degree $1 - a_k$ in $x_0$ and $x_1$ which necessarily satisfy $\sum_{i=2}^n c_{ik} f_i = 0$. Then for $v$ to lie in $T_LF_{\vec{a}}(X)$, there must exist $c_{ik}'$ of degree $1-a_k$ such that
\[\sum_{i=2}^n(c_{ik} + \epsilon c_{ik}')\left(f_i + \epsilon \sum_{j=2}^n H_{ij} v_j\right) = 0.\]
Equivalently, this shows the following.

\begin{Lemma} \label{tlem}
A tangent vector $v \in T_LF(X) \cong H^0(N_{L/X}) \subset H^0(\O(1)^{n-1})$ is in the subspace $T_LF_{\vec{a}}(X)$ if and only if for each $k$,
\begin{equation} \label{tc}
\sum_{i=2}^n \sum_{j=2}^n c_{ik} H_{ij} v_j \in \mathrm{span}\{f_2, \ldots, f_n\} \text{ with coefficients of degree $1-a_k$}.
\end{equation}
\end{Lemma}

This appears to be a somewhat complicated condition, but with nice choices of the $c_{ik}$ and $f_i$, the linear conditions in the lemma become clear. We use this strategy to prove the following.

\begin{Theorem} \label{tada}
If $u(\vec{a}) \leq 2n - d - 3$, then $F_{\vec{a}}(X)$ has codimemsion $u(\vec{a})$ for general $X$.
\end{Theorem}
\begin{proof}
First note that if $a_1 \leq -2$, then $u(\vec{a})\geq h^0(\O(\vec{a})) > 2n - d - 3$, so we may restrict our attention to lines $L$ with
\[N_{L/X} \cong \O(-1)^{b} \oplus \O^{\ell} \oplus \O(1)^{m}.\]
Since $h^1(N_{L/X}) = 0$, such a line is always a smooth point of $F(X)$ with 
\[\dim T_LF(X) = 2n - d - 3.\]
The expected codimemension of lines with this splitting type is $bm$, so we can assume $bm \leq 2n - d - 3$.
To prove the theorem, it suffices to find $L \subset X$ with this normal bundle where $T_L F_{\vec{a}}(X) \subset T_LF(X)$ has codimension at least $bm$. Then upper semicontinuity of dimension of fibers together with Theorem \ref{uni-ecd} will imply that $\codim F_{\vec{a}}(X) = u(\vec{a})$ for general $X$.

Now, consider the map
\[\O(-1)^b \oplus \O^{\ell} \oplus \O(1)^m \rightarrow \O(1)^{n-1}\]
given by
\[
\setcounter{MaxMatrixCols}{20}
A = \left(\begin{matrix} x_0^2 & 0 & \cdots  & 0 & 0 & 0 & \cdots & 0 & 0 & 0 & \cdots & 0\\
x_1^2 & x_0^2 & \cdots & 0 & 0  & 0 & \cdots & 0 & 0 & 0 & \cdots & 0\\
0 & x_1^2 & \ddots & \vdots & \vdots & \vdots && \vdots & \vdots & \vdots && \vdots\\
\vdots & \vdots &  \ddots & x_0^2 & 0  & 0 & \cdots & 0 & 0 & 0 & \cdots & 0\\
0 & 0 &\cdots & x_1^2 & x_0 & 0 & \cdots & 0 & 0 & 0 & \cdots & 0\\
0 & 0 & \cdots & 0 & x_1 & x_0 & \cdots & 0 & 0 & 0 & \cdots & 0\\
0 & 0 & \cdots & 0 & 0 & x_1 & \ddots & \vdots & \vdots & \vdots && \vdots\\
\vdots & \vdots & & \vdots & \vdots & \vdots & \ddots & x_0 & 0  & 0 & \cdots & 0\\
0 & 0 & \cdots & 0 & 0 & 0 & \cdots & x_1 & 0 & 0  & \cdots & 0\\
0 & 0 & \cdots & 0 & 0 & 0 & \cdots & 0 & 1 & 0  & \cdots & 0\\
0 & 0 & \cdots & 0 & 0 & 0 & \cdots & 0 & 0 & 1 & \cdots & 0 \\
\vdots & \vdots && \vdots & \vdots & \vdots && \vdots & \vdots & \vdots & \ddots & \vdots \\
\undermat{b}{0 & 0 & \cdots & 0 &} \undermat{\ell}{0 & 0 & \cdots & 0 &}\undermat{m}{ 0 & 0 & \cdots & 1}
\end{matrix}\right)
\]

\vspace{.15in}
\noindent
The map to the cokernel of this inclusion is given by the maximal minors of $A$. If $A_i$ is $A$ with the $i$th row deleted, then we have
\[m_i = (-1)^i \det A_i = (-1)^i\begin{cases} x_0^{2i-2} x_1^{d-2i+1} & i \leq b \\ x_0^{b+i-1} x_1^{d-b-i} & b+1 \leq i \leq b + \ell\\ 0 & b+\ell+1 \leq i \leq n-2.\end{cases} \]
Since the $m_i$ have no common zeroes, we can apply Bertini's theorem to find some smooth hypersurface $X = V(f)$ containing $L$ such that $f_{i+1} = m_i$ are the partial derivatives of $f$ restricted to $L$.
For $b+\ell+2 \leq k \leq n-1$, we have $c_{i,k-1} = \delta_{i,k}$, so the condition \eqref{tc} that $v \in T_LF_{\vec{a}}(X)$ becomes that the coefficient of $x_0^e x_1^{d-e-1}$ in 
\begin{equation} \label{here}
\sum_{j=2}^n H_{kj}(x_0, x_1) v_j(x_0, x_1) 
\end{equation}
must vanish for odd $e < 2b$. The general form of an element 
\[v \in T_LF(X) = H^0(N_{L/X}) \subset H^0(\O(1)^{n-1})\]
is
\begin{align*}
v = (\underbrace{0, \ldots, 0}_{j=2,\ldots,b+1}, \underbrace{\lambda_1x_0, \lambda_1 x_1 + \lambda_2 x_0, \ldots, \lambda_{\ell-1} x_1 + \lambda_\ell x_0, \lambda_\ell x_1,}_{j=b+2, \ldots, b+\ell+1} \\
\underbrace{ \mu_1^{(0)} x_0 + \mu_1^{(1)}x_1, \ldots, \mu_m^{(0)} x_0 + \mu_m^{(1)} x_1}_{j=b+\ell+2, \ldots, n-2}).
\end{align*}
If we can find a smooth $X$ with appropriate $H_{kj}$, then setting the coefficient of $x_0^e x_1^{d-e-1}$ in \eqref{here} to zero will give an independent linear condition on the coordinates $\lambda_i$ and $\mu_i^{(j)}$ for each of the $b$ possible values of $e$ and $m$ possible values of $k$, showing that $\codim(T_LF_{\vec{a}}(X) \subset T_LF(X)) = bm$. Suppose that
\[H_{kj}(x_0, x_1) = \begin{cases} x_1^{d-2} + x_0^3 x_1^{d-5} & \text{if $j = k \geq b+\ell+2$}\\ x_0^{2i}x_1^{d-2i-2} & \text{if } j -b = i + (k-\ell-b-2)(b-2) \\ 0 & \text{otherwise.}  \end{cases} \]
Above, if the exponent on a variable is negative, it is understood that we omit that term.
With this choice, setting the coefficient of $x_0x_1^{d-2}$ to zero in \eqref{here} forces each $\mu_k^{(0)} = 0$. Similarly, if $e = 3< 2b$, setting the coefficient of $x_0^3x_1^{d-4}$ to zero forces each $\mu_k^{(1)} = 0$. Then for the remaining $b-2$ values $e  = 5,\ldots, 2b - 1$ and each $b + \ell + 2 \leq k \leq n-1$, the only contribution to $x_0^e x_1^{d-e-1}$ comes from the
\[j = b + \frac{e-1}{2} + (k-\ell-b-2)(b-2)\]
term of the sum in \eqref{here}, when $\lambda_j x_0$ meets $H_{kj} = x_0^{e-1}x_1^{d-e-1}$. This gives $m(b-2)$ more conditions $\lambda_j = 0$.
Finally, Bertini's theorem guarantees that there exists a smooth hypersurface with these $f_i$ and $H_{kj}$, so we are done.
\end{proof}

\section{Important examples}
\subsection{Cubic hypersurfaces}

As we have seen, there are only two possible splitting types for the normal bundle of a line $L$ on a cubic hypersurface $X$: either
\begin{align*}
N_{L/X} &= \O \oplus \O \oplus \O(1)^{n-4} & \quad & \text{(balanced)} \\
\intertext{or}
\qquad N_{L/X} &= \O(-1) \oplus \O(1)^{n-3} &\quad &\text{(unbalanced).}
\end{align*}
Since $h^0(N_{L/X}) = 2n - 6$ for either splitting type, the Fano scheme is always smooth of that dimension.
To simplify notation, let $F' = F_{(-1, 1, \ldots, 1)}(X)$ be the locus of unbalanced lines. The expected codimension of $F'(X)$ is $n-3$.

Specializing our analysis in the previous section gives rise to a concrete description of the tangent space to $F'(X)$. Given a line $L \in F'(X)$, we can always choose coordinates so that $L = V(x_2, x_3, x_4)$. Looking at our favorite short exact sequence of normal bundles,
\[0 \rightarrow N_{L/X} \rightarrow \O(1)^{n-1} \xrightarrow{(f_2, \ldots, f_n)} \O(d) \rightarrow 0,\]
where $f_i = \left.\frac{\partial f}{\partial x_i}\right|_L$, we see that
\[L \in F'(X) \quad \Longleftrightarrow \quad \dim \mathrm{span}\{f_2, \ldots, f_n\} = 2.\] 
Given $L \in F'(X)$, we can thus choose coordinates so that $\frac{\partial f}{\partial x_i} = 0$ for all $i \geq 4$. Next, note that the span of two quadratic polynomials with no common zeros contains exactly two squares. Hence, we can choose coordinates $x_0$ and $x_1$ on $L$ so that the two squares in the span of $f_2$ and $f_3$ are $x_0^2$ and $x_1^2$. Finally, after a possible linear change of variables between $x_2$ and $x_3$, 
we can write
\begin{equation} \label{fform}
f = x_2x_0^2 + x_3x_1^2 + \sum_{i \geq j} h_{ij}(x_0, x_1) x_i x_j + (x_2, x_3, x_4)^3.
\end{equation}

With this set up, the tangent space to the Fano scheme is simply
\[T_LF(X) = \{v = (v_2, \ldots, v_n) \in H^0(\O(1)^{n-1}) \cong T_L\gg(1, n): v_2 = v_3 = 0\}.\]
Then Lemma \ref{tlem} tells us that the condition for $v = (0,0,v_4, \ldots, v_n)$ to be in the tangent space is that the coefficient of $x_0x_1$ in $\sum_{j} H_{ij}(x_0, x_1)v_j(x_0, x_1)$ vanishes for each $i \geq 4$.
Writing $H_{ij}(x_0, x_1) = H_{ij}^{(0)}x_0 + H_{ij}^{(1)} x_1$ and identifying coordinates on $T_LF(X)$ as the coefficients $v_j^{(0)}$ and $v_j^{(1)}$ for $j \geq 4$, we have shown the following.
\begin{Lemma} \label{tlf}
In our chosen coordinates,
\[T_LF'(X) = \ker \left(\begin{matrix} H_{44}^{(1)} & H_{44}^{(0)} & \cdots &  H_{4n}^{(1)} & H_{4n}^{(0)} \\
\vdots & \vdots & \ddots & \vdots & \vdots \\
H_{4n}^{(1)} & H_{4n}^{(0)} & \cdots & H_{nn}^{(1)} & H_{nn}^{(0)} \end{matrix}\right) \subset T_LF(X).\]
The codimension of $T_LF' \subset T_LF$ is thus the rank of this matrix. 
\end{Lemma}

Because $\Sigma_{(-1, 1, \ldots, 1)}$ is smooth, Sard's theorem tells us that $F'(X)$ is smooth in general. However, for some $X$, the locus $F'(X)$ may be singular, and we can ask: what sorts of singularities occur? We use the above description of the tangent space to $F'(X)$ to answer this question for cubic threefolds.

\begin{Theorem}
For  $X \subset \pp^4$ a smooth cubic hypersurface, $F' = F'(X)$ is a curve that has at worst nodes as singularities.
\end{Theorem}

\begin{proof}
Suppose $L \in F'$ and $\dim T_LF' = 2$. Then $T_LF' = T_LF$, so after choosing coordinates so that $f$ has the form in \eqref{fform}, Lemma \ref{tlf} tells us that $h_{44} = 0$. Geometrically, this says that the intersection of $X$ with the $2$-plane defined by $x_2 = x_3 = 0$ is $3L$.

We now use the explicit local equations Example \ref{cubie4} to compute the tangent cone to $F'$ at $L$. Let $a_{ij}$ be local coordinates on $\gg(1, 4)$ as in that section.
First, note that the Fano scheme $F(X)$ is cut out by $4$ polynomials. These polynomials are the coefficients of the polynomial in $x_0$ and $x_1$ that results from plugging $x_i = a_{i0} x_0 + a_{i1} x_1$ for $i = 2, 3, 4$ into $f$. With $f$ having the form \eqref{fform}, the leading terms of these four polynomials are $a_{20}, a_{21}, a_{30},$ and $a_{31} $, so the tangent plane to the Fano scheme is 
\[T_LF(X) = V(a_{20}, a_{21}, a_{30}, a_{31}).\]
Note also that since $h_{44} = 0$, the quadratic terms of these four polynomials are in the ideal $I = \langle a_{20}, a_{21}, a_{30}, a_{31}\rangle.$
 
The locus $F' \subset \gg(1, 4)$ is cut out by the equations of $F$ together with the determinant of the matrix.
\[C(-1) = \left(\begin{matrix}C_{20} & C_{30} & C_{40} \\
C_{21} & C_{31} & C_{41} \\
C_{22} & C_{32} & C_{42} \\
\end{matrix}\right)\]
where $C_{ij}$ are polynomials in $a_{k\ell}$ defined implicitly by
\[\frac{\partial f}{\partial x_i}(x_0, x_1, a_{20}x_0+a_{21}x_1, a_{30}x_0+a_{31}x_1, a_{40}x_0+a_{41}x_1) = C_{i0} x_1^2 + C_{i1} x_1x_0 + C_{i2} x_0^2.\]
 We claim that, modulo the ideal $I$, the leading term of $\det C(-1)$ is a multiple of $a_{40}a_{41}$.
First, observe that $C_{22}$ and $C_{30}$ are the only entries that begin with constant terms, the rest being linear or higher order in the $a_{ij}$. Next, we have
\[\frac{\partial f}{\partial x_4} = \lambda x_4^2 + (x_2, x_3),\]
so the linear terms of the entries $C_{4i}$ in the last column are all in the ideal $I$. 

Putting these observations together, we see that the leading term of $\det C(-1)$ modulo $I$ is $C_{41}$ mod $I$, which is $2 \lambda a_{40}b_{41}$. 
Thus, $I(F')$ can be generated by five polynomials whose leading terms are $a_{20}, a_{21}, a_{30}, a_{31}$, and $a_{40}a_{41}$. The leading term of any polynomial in $I(F')$ is therefore in the ideal generated by these leading terms. This shows that the tangent cone to $L \in F'$ is a union of the two distinct lines $V(a_{40})$ and $V(a_{41})$ in the tangent plane, so $L\in F'$ must be a node.
\end{proof}

\subsection{The Fermat quartic threefold} \label{fsec}
Although we know that $F_{\vec{a}}(X)$ has the expected dimension for general $X$, sometimes $F_{\vec{a}}(X)$ has larger dimension for certain $X$.

One illuminating example of this is when $X = V(x_0^4 + \ldots + x_4^4) = V(f) \subset \pp^4$ is the Fermat quartic threefold. If $H$ is a hyperplane defined by $x_i = \zeta x_j$ for $\zeta$ a $4$th root of $-1$, then the hyperplane section $H \cap X$ is a cone over a smooth quartic curve in $\pp^2$. We claim that for any line $L$ on this cone, we have $N_{L/X} = \O(-2) \oplus \O(1)$. After possibly permuting coordinates (which preserves the equation of $X$), we can assume $L$ has the form 
\begin{equation}\label{above}
L = \{[x_0, x_1, x_0a, x_0b, x_1\zeta]:[x_0, x_1] \in \pp^1\}
\end{equation}
for some $a, b \in \cc$ with $a^4+b^4 +1 = 0$. From this, we see that
\begin{equation}
\mathrm{span}\left\{\left.\frac{\partial f}{\partial x_i}\right|_L\right\} = \langle x_0^3, x_1^3\rangle,
\end{equation}
which is two-dimensional. In particular, the right hand map in the sequence 
\begin{equation*}
\begin{tikzcd}
&0 \arrow{r} &N_{L/X} \arrow{r} &\O(1)^3 \arrow{r} & \O(4) \arrow{r} & 0
\end{tikzcd}
\end{equation*}
is given by a collection of three necessarily linearly dependent polynomials, implying $\O(1)$ includes into its kernel, $N_{L/X}$. It follows by degree considerations that $N_{L/X} = \O(-2) \oplus \O(1)$. Because $h^0(N_{L/X}) = 2$, the family of lines on each of these hyperplane sections corresponds to a nonreduced quartic curve on $F(X)$.
Since there are $40$ such hyperplanes --- corresponding to the ${5 \choose 2}$ choices of two coordinates $x_i, x_j$ and $4$ choices of $\zeta$ --- computing the class $[F(X)] = c_5(\mathrm{Sym}^4\mathcal{S}^*) = 320\sigma_{3,2}$ in the Chow ring of the Grassmannian shows that $F(X)$ is exactly the union of these $40$ non-reduced curves.

This is important example in that \textit{every} line on $F(X)$ has normal bundle $\O(-2) \oplus \O(1)$. That is, the underlying sets of $F(X)$ and $F_{(-2, 1)}(X)$ are equal. Nevertheless, $F(X)$ and $F_{(-2,1)}(X)$ are \textit{not} equal as schemes. In fact, $F_{(-2,1)}(X)$ is the reduced subscheme of $F(X)$, as can be seen from computing the tangent space $T_LF_{(-2, 1)}(X)$ for general $L \in F(X)$. Suppose $L$ has the form in \eqref{above}. Let
\[w_2 = x_2 - ax_0, \qquad w_3 = x_3 - bx_0, \qquad \text{and} \qquad w_4 = x_4 - \zeta x_1.\]
Then $L = V(w_2, w_3, w_4)$, and we can rewrite the defining equation of $X$ in these coordinates as
\begin{align*}
&x_0^4+x_1^4 + (w_2 + ax_0)^4 + (w_3+bx_0)^4 + (w_4 + \zeta x_1)^4 \\
&\qquad \qquad \qquad = 4w_2a^3x_0^3 + 4w_3b^3x_0^3 + 4w_4\zeta^3x_1^3 \\
&\qquad \qquad \qquad \quad + 6w_2^2a^2x_0^2 + 6w_3^2b^2x_0^2 + 6w_3^2\zeta^2x_1^2 + (w_2, w_3, w_4)^3.
\end{align*}
This tells us that
\[T_LF(X) = \{(v_2, v_3, v_4): v_4 = 0, a^3v_2 + b^3 v_3 = 0\}.\]
The inlcusion $N_{L/X} \rightarrow \O(1)^3$ can thus be represented by
\[\left(\begin{matrix} 0 & 0 \\ 0 & b^3 \\ 0 & -a^3 \end{matrix}\right),\]
so Lemma \ref{tlem} tells us that $T_LF'(X) \subset T_LF(X)$ is determined by the condition
\begin{equation} \label{xeq}
x_0^2(b^3a^2 v_2 - a^3b^2v_3) \in \mathrm{span}\{x_0^3, x_1^3\}.
\end{equation}
This is equivalent to
\[b^3a^2v_2^{(1)} - a^3b^2 v_3^{(1)} = 0.\]
If $a$ and $b$ are both non-zero, then plugging in $v_3^{(1)} = -\frac{a^3}{b^3} v_2^{(1)}$ we find
\begin{align*}
0 = v_2^{(1)}\left(b^3a^2 - a^3 b^2\left(-\frac{a^3}{b^3}\right)\right) = v_2^{(1)} \left(\frac{a^2}{b}\right)(b^4 + a^4) \quad &\Rightarrow \quad v_2^{(1)} = 0  \\
&\Rightarrow \quad v_3^{(1)} = 0.
\end{align*}

Thus, for general lines $L \in F(X)$ we have $\dim T_LF'(X) = 1$, showing that $F'(X)$ is reduced. On the other hand, if one of $a$ or $b$ is zero, then $T_LF'(X) = T_LF(X)$ because the left-hand side of \eqref{xeq} is zero, so the condition is automatically satisfied. These lines $L$ correspond points where two components of $F(X)$ intersect.

\section{The class of $\overline{F_{\vec{a}}(X)}$} \label{st3}
For general $X$, knowing that $F_{\vec{a}}(X)$ has the expected codimension in $F(X)$ allows us to compute the class of $\overline{F_{\vec{a}}(X)}$ in the Chow ring of $\gg(1, n)$. As noted at the beginning of the proof of Theorem \ref{tada}, for general $X$, $F_{\vec{a}}(X)$ is only nonempty when
\[\vec{a} = (-1, \ldots, -1, 0, \ldots, 0, 1, \ldots, 1) = (-1^b, 0^\ell, 1^m),\]
where $-b+m = n-d-1$ and $b+\ell+m = n-2$.

\begin{Proposition} \label{classp}
Let $\vec{a} = (-1^b, 0^\ell, 1^m)$ as above. Then for general $X$, the class of $\overline{F_{\vec{a}}(X)}$ in the Chow ring of the Grassmannian is
\[[\overline{F_{\vec{a}}(X)}]  = [F(X)] \cdot \Delta^{m}_{b}\left(\frac{c(\mathcal{Q})}{c(\mathrm{Sym}^{d-1} \mathcal{S}^*)}\right) = c_{d+1}(\mathrm{Sym}^{d} \mathcal{S}^*) \cdot \Delta^{m}_{b}\left(\frac{c(\mathcal{Q})}{c(\mathrm{Sym}^{d-1} \mathcal{S}^*)}\right), \]
where $\Delta^m_b$ is defined in \eqref{ddelta}.
\end{Proposition}
\begin{proof}
Restricting the vector bundles $\pi_*\E(i)$ and $\pi_*\cF(i)$ from Section \ref{clsec} to the fibers $F(X)$, we may realize $\overline{F_{\vec{a}}(X)}$ as 
\[\overline{F_{\vec{a}}(X)} = M_{r_{-1}(\vec{a})}(\pi_*\phi(-1)|_{F(X)}).\]
We have
\[\pi_*\E(i)|_{F(X)} = \mathcal{Q}|_{F(X)} \otimes \mathrm{Sym}^{1+i} \mathcal{S}^*|_{F(X)} \qquad \text{and} \qquad \pi_* \cF(i) = \mathrm{Sym}^{d+i} \mathcal{S}^*|_{F(X)},\]
so if $\iota: F(X) \rightarrow \gg(1, n)$ denotes the inclusion of $F(X)$ in the Grassmannian, the class of $F_{\vec{a}}(X)$ in the Chow ring of $F(X)$ is 
\[\Delta^m_b\left(\frac{\iota^* c(\mathcal{Q})}{\iota^*c(\mathrm{Sym}^{d-1}\mathcal{S}^*)}\right) = \iota^*\Delta^m_b\left(\frac{ c(\mathcal{Q})}{c(\mathrm{Sym}^{d-1}\mathcal{S}^*)}\right).\]
Finally, the push-pull formula tells us that the pushforward of this class to the Chow ring of the Grassmannian is given by intersecting with $[F(X)]$.
\end{proof}

\begin{Example}[$n=d=5$]
For general quintic hypersurfaces $X \subset \pp^5$, we have $\dim F(X) = 2$. The following table lists the possible splitting types for the normal bundle and the expected codimension of lines with that splitting type

\vspace{.03in}
\begin{center}
\begin{tabular}{c|c|c|c|c}
$\vec{a}$ &$ (-3, 1, 1)$ & $(-2, 0, 1)$ & $(-1, -1, 1)$ & $(-1, 0, 0)$ \\[0.5ex]
\hline
\rule{0pt}{2.5ex} $u(\vec{a})$ & 4 & 3 & 2 & 0 
\end{tabular}
\end{center}

\vspace{.05in}
\noindent
Thus, Theorem \ref{tada} says that a general $X$ contains no lines with splitting type $(-3, 1, 1)$ or $(-2, 0, 1)$ and finitely many lines with splitting type $(-1, -1, 1)$. So we can ask: how many unbalanced lines are there? Since $\Sigma_{(-1, -1, 1)}$ is smooth, Sard's theorem tells us that $F_{(-1, -1, 1)}(X)$ is smooth, and hence reduced, for general $X$. Thus, the answer is the degree of the class given in Proposition \ref{classp}:
\begin{align*}
[F_{(-1, -1, 1)}(X)] = [\overline{F_{(-1, -1, 1)}(X)}] &= c_6(\mathrm{Sym}^5 \mathcal{S}^*) \cdot \Delta^1_2\left(\frac{c(\mathcal{Q})}{c(\mathrm{Sym}^4\mathcal{S}^*)}\right).
\end{align*}
Here, $\Delta^1_2$ means we need the degree $2$ piece of the quotient in the brackets. Writing
\begin{align*}
\frac{c(\mathcal{Q})}{c(\mathrm{Sym}^4\mathcal{S}^*)} = (1 + \sigma_1 + \sigma_2 + \ldots)(1 - (c_1 + c_2 + \ldots) + (c_1 + c_2 + \ldots)^2 - \ldots)
\end{align*}
where $c_i = c_i(\mathrm{Sym}^4\mathcal{S}^*)$, we see that the degree $2$ part is
\[c_1^2 + c_2 + \sigma_2 - c_1 \sigma_1.\]
Using the splitting principle, we calculate
\[c_1 = 10 \sigma_1 \qquad \text{and} \qquad c_2 = 35\sigma_1^2 + 20 \sigma_{1,1} = 35\sigma_2 + 55\sigma_{1,1},\]
so
\begin{align*}
\Delta^1_2 \left(\frac{c(\mathcal{Q})}{c(\mathrm{Sym}^4\mathcal{S}^*)}\right) &= 100\sigma_1^2 + 35\sigma_2 + 55\sigma_{1,1} + \sigma_2 - 10\sigma_1^2 = 126 \sigma_2 + 145\sigma_{1,1}.
\end{align*}
Another application of the splitting principle shows that
\[c_6(\mathrm{Sym}^5 \mathcal{S}^*) = 600\sigma_{1,1}\sigma_1^4 + 1450 \sigma_{1,1}^2\sigma_1^2 - 225\sigma_{1,1}^3 = 3250\sigma_{4,2} + 2425\sigma_{3,3}.\]
Putting this all together, we find
\[[F_{(-1, -1, 1)}(X)] =  (126\sigma_2 + 145 \sigma_{1,1})(3250\sigma_{4,2} + 2425\sigma_{3,3}) = 761125\sigma_{4,4}.\]
That is, there are $761125$ unbalanced lines on a general quintic fourfold.

\end{Example}

\section{Dimension bounds for certain $F_{\vec{a}}(X)$} \label{bounds}

In Section 5, we determined the dimension of $F_{\vec{a}}(X)$ for general $X$. In this section, we give upper bounds on the dimension of lines with two particular splitting types which are valid for all smooth hypersurfaces.

\subsection{Completely unbalanced lines} \label{cul}
There is a unique splitting type $\vec{a}$ which is the ``most unbalanced" allowed by the conditions $a_i \leq 1$ and $\sum_{i=1}^{n-2} = n-d-1$, namely $\vec{a} =(-d+2, 1, \ldots, 1)$.
For a given degree $d$, let $F'(X) = F_{(-d+2, 1,\ldots,1)}(X)$ denote the locus of lines with this splitting type.
The expected codimension of lines with this splitting type is $(n-3)(d-2)$, so Theorem \ref{uni-ecd} tells us 
\[\dim \Sigma_{(-d+2, 1, \ldots, 1)} = 2n - d - 3 + N - (n-3)(d-2).\]
It follows that
\begin{equation} \label{F'bound}
\dim F'(X) \geq 2n - d - 3 - (n-3)(d-2)  = n-3 - (d-3)(n-2),
\end{equation}
and if $d \geq 4$, the locus $F'(X)$ is empty for general $X$. However, for certain $X$, a priori, $F'(X)$ could be as large as $F(X)$, which is in turn only bounded above by $\dim T_LF(X) = 2n - 6$. The following provides an upper bound of half this amount.

\begin{Theorem} \label{dimF'}
If $X \subset \pp^n$ is a smooth hypersurface of degree $d \geq 3$, then $\dim F'(X) \leq n-3$.
\end{Theorem}

\begin{remark}
This is proved for cubics in Corollary 7.6 of \cite{CG}.
\end{remark}

Note that for cubics, the lower bound \eqref{F'bound} then determines the dimension of $F'(X)$ exactly.
Given that for $d \geq 4$, the lower bound in \eqref{F'bound} is trivial, one might expect that Theorem \ref{dimF'} is far from the truth of what is actually achieved. However, the bound is sharp, as demonstrated by the following proposition.

\begin{Proposition} \label{fermat}
Let $X = V(x_0^d + \ldots + x_n^d)$ for $d\geq 4$. Then $\dim F'(X) = n-3$.
\end{Proposition}

\begin{proof}[Proof of Theorem \ref{dimF'}]
Suppose that we have some $L \in F'(X)$. Choose coordinates so that $L = V(x_2, \ldots, x_n)$ and write the defining equation of $X$ as
\[f = \sum_{i=2}^n x_i f_i(x_0, x_1) + (x_2,\ldots, x_n)^2.\]
In the short exact sequence of normal bundles
\[0 \rightarrow \O(-d+2) \oplus \O(1)^{n-3} \rightarrow \O(1)^{n-1} \xrightarrow{(f_2, \ldots, f_n)} \O(d) \rightarrow 0,\]
each inclusion of $\O(1)$ defines an independent linear relation of $f_2, \ldots, f_n$. Hence, the span of $f_2, \ldots, f_n$ is $2$-dimensional. The other partial derivatives $\frac{\partial f}{\partial x_0}$ and $\frac{\partial f}{\partial x_1}$ vanish along $L$, so the Gauss map
\[\G: X \rightarrow \pp^{n*}, \quad p \mapsto \left[\frac{\partial f}{\partial x_0}(p), \ldots, \frac{\partial f}{\partial x_n}(p)\right]\]
sends the line $L$ with degree $d-1 \geq 2$ onto a line. 

Now consider the incidence correspondence
\[\Theta = \{(p, q, L) \in X \times X \times F'(X): p, q \in L, \G(p) = \G(q), p \neq q\}.\]
Let $\pi_1 : \Theta \rightarrow X$ be projection onto the first factor, and $\pi_3: \Theta \rightarrow F'(X)$ projection onto the last factor. Every point in $\pi_1(\Theta) \subset X$ is a point where the fiber of the Gauss map consists of $2$ or more points.
Since the Gauss map is generically one-to-one,
$\overline{\pi_1(\Theta)}$ cannot be all of $X$. Hence, 
\[\dim \pi_1(\Theta) \leq \dim(X) - 1 = n-2.\]
Meanwhile, since the Gauss map is finite,
the fibers of $\pi_1$, which consist of points with the same image under $\G$, must be finite. It follows that $\dim \Theta \leq n-2$ as well. Finally, our work in the previous paragraph shows that the fibers of $\pi_3$ are one-dimensional, so we can conclude that $\dim F'(X) \leq n-3$.
\end{proof}

\begin{proof}[Proof of Proposition \ref{fermat}]
Let $X = V(x_0^d + \ldots + x_n^d)$ and let $H$ be the hyperplane defined by $x_n = \zeta x_{n-1}$ where $\zeta^d = -1$. Then $H \cap X$ is a cone over the smooth hypersurface $V(x_0^d + \ldots + x_{n-2}^d) \subset \pp^{n-2}$. Let $L \subset X \cap H$ be any line passing through the vertex of this cone. Then any line in the tangent plane along $L$ to $X \cap H$ is on $X \cap H$ to first order. Hence, 
\[\dim T_LF(X) \geq \dim \gg(1, n-2) = 2n - 6.\]
For $d \geq 4$, the only way that $h^0(N_{L/X})$ can be greater than or equal to $2n-6$ is when $N_{L/X}$ has $n-3$ copies (the maximal number) of $\O(1)$'s. Thus, all lines on $X \cap H$, which form a family of dimension at least $n-3$, are in $F'(X)$. Hence, $\dim F'(X) \geq n-3$, and Theorem \ref{dimF'} shows that we have equality.
\end{proof}

\subsection{Almost completely unbalanced lines}
A similar trick involving the Gauss map works to find an upper bound on the dimension of the locus of lines whose normal bundle has the form
\[N_{L/X} \cong \O(-a) \oplus \O(-b) \oplus \O(1)^{n-4}, \qquad a+b = d-3.\]
Let $G_{a,b}(X) = F_{(-a,-b,1,\ldots,1)}(X)$ denote the locus of such lines.
The expected codimension of $G_{a, b}(X)$ is
\[(n-4)(d-3) + \mathrm{max}\{0, a-b-1\},\]
giving rise to the lower bound on the dimension
\[\dim G_{a,b}(X) \geq (n-3) - (d-4)(n-3) - \mathrm{max}\{0, a - b - 1\}.\]
Note that for cubics, $a$ and $b$ are both forced to be zero and these are the balanced lines. In this case, this lower bound coincides with the dimension of the tangent space $T_LF(X)$, so $\dim G_{0,0} = 2n - 6$. For $d \geq 4$, we have the following upper bound on the dimension.

\begin{Theorem} \label{gbound}
If $d \geq 4$, then $\dim G_{a, b}(X) \leq n-1$ for all smooth hypersurfaces $X \subset \pp^n$ of degree $d$.
\end{Theorem}

\begin{proof}
Suppose that we have some $L \in G_{a,b}(X)$. Choose coordinates so that $L = V(x_2, \ldots, x_n)$ and write the defining equation of $X$ as
\[f = \sum_{i=2}^n x_i f_i(x_0, x_1) + (x_2,\ldots, x_n)^2.\]
In the short exact sequence of normal bundles
\[0 \rightarrow \O(-a) \oplus \O(-b) \oplus \O(1)^{n-4} \rightarrow \O(1)^{n-1} \xrightarrow{(f_2, \ldots, f_n)} \O(d) \rightarrow 0,\]
each inclusion of $\O(1)$ defines an independent linear relation of $f_2, \ldots, f_n$. Hence, the span of $f_2, \ldots, f_n$ is $3$-dimensional. The other partial derivatives $\frac{\partial f}{\partial x_0}$ and $\frac{\partial f}{\partial x_1}$ vanish along $L$, so now the Gauss map
\[\G: X \rightarrow \pp^{n*}, \quad p \mapsto \left[\frac{\partial f}{\partial x_0}(p), \ldots, \frac{\partial f}{\partial x_n}(p)\right]\]
sends the line $L$ with degree $d-1 \geq 3$ onto a plane curve. Since the genus of a smooth degree $e$ plane curve is ${e-1 \choose 2} > 0$ for $e \geq 3$, such a map cannot be an embedding. Hence, there exists some point on $L$ where the differential is zero, or two points on $L$ that have the same image (or both).

First consider the incidence correspondence
\[\Theta = \{(p, L) \in X \times G_{a,b}(X): p \in L \text{ and } (d\G|_L)_p = 0\}\]
with projections $\pi_1 : \Theta \rightarrow X$ and $\pi_2: \Theta \rightarrow G_{a,b}(X)$. If $p \in X$ has $\dim \pi_1^{-1}(p) = k$, then $\dim \ker (d\G)_p \geq k$.
The dimension of the locus of points in $X$ where the differential of the Gauss map has this rank is at most $n-1-k$. We conclude that $\dim \Theta \leq n-1$, and so $\dim \pi_2(\Theta) \leq n-1$ as well.

Next consider the incidence correspondence
\[\Theta' = \{(p, q, L) \in X \times X \times G_{a,b}(X): p, q \in L, \G(p) = \G(q), p \neq q\}.\]
Because the image of $\Theta'$ onto the first factor consists of points where $\G$ fails to be one-to-one, we must have $\dim \pi_1(\Theta') \leq \dim(X) - 1 = n-2$. On the other hand, the fact that the Gauss map is finite guarantees that the fibers $\pi_1$ are finite, so $\dim \Theta' \leq n-2$, and hence $\dim \pi_3(\Theta') \leq n-2$.

The fact that $\G|_L$ is not an embedding for any $L \in G_{a,b}(X)$ tells us that $G_{a,b}(X) = \pi_2(\Theta) \cup \pi_3(\Theta')$. In particular, we have
\[\dim G_{a,b}(X) = \mathrm{max}\{\dim \pi_2(\Theta) , \dim \pi_3(\Theta')\} \leq n-1,\]
which is the desired result.
\end{proof}

\end{document}